\documentclass[reqno]{amsart}
\UseRawInputEncoding
\makeatletter
\@namedef{subjclassname@2020}{\textup{2020} Mathematics Subject Classification}
\makeatother

\usepackage{amssymb}
\usepackage[all]{xy}
\usepackage{enumitem}
\usepackage{hyperref}
\usepackage{graphicx,xcolor}
\usepackage{tikz}
\usepackage{tikz-cd}

\def\commentout#1{{}}

\def\C{{\mathbb C}}
\def\N{{\mathbb N}}

\def\bI{{\mathbf I}}
\def\bJ{{\mathbf J}}

\def\profin#1{{{#1}^{{\kern -0.15em}\wedge}}}

\def\Xhat{{\profin{X}}}
\def\Ihat{{\profin{I}}}
\def\Jhat{{\profin{J}}}

\def\gpr{{P}}

\newcommand\Aut{\operatorname{Aut}}

\newcommand\parto{\rightharpoonup}
\newcommand\id{{\operatorname{id}}}

\newcommand\pr{{\operatorname{pr}}}

\newcommand\lex{{\operatorname{lex}}}
\newcommand\lexprod{{\leftthreetimes}}
\newcommand\lcm{{\operatorname{lcm}}}
\newcommand\Sets{{\operatorname{Sets}}}
\newcommand\Iso{{\operatorname{Iso}}}

\newtheorem{theorem}{Theorem}[section]
\newtheorem{lemma}[theorem]{Lemma}
\newtheorem{proposition}[theorem]{Proposition}
\newtheorem{corollary}[theorem]{Corollary}

\theoremstyle{definition}
\newtheorem{definition}[theorem]{Definition}

\theoremstyle{remark}
\newtheorem{remark}[theorem]{Remark}

\numberwithin{equation}{section}

\usepackage{amsmath}	
\begin{document}
\title[Wreath products and non Schurian association schemes]
{Wreath products and projective system of
non Schurian association schemes}

\author{Makoto Matsumoto}
\address{Mathematics Program \\ 
Graduate School of Advanced Science and Engineering\\
Hiroshima University, 739-8526 Japan}
\email{m-mat@math.sci.hiroshima-u.ac.jp}

\author{Kento Ogawa}
\address{Mathematics Program \\ 
Graduate School of Advanced Science and Engineering\\
Hiroshima University, 739-8526 Japan}
\email{knt-ogawa@hiroshima-u.ac.jp}


\keywords{Association Scheme, wreath product, 
non-Schurian association scheme, S-ring,
profinite association scheme, kernel scheme
}
\thanks{
The first author is partially supported by JSPS
Grants-in-Aid for Scientific Research
JP26310211 and JP18K03213.
The second author is partially supported by
JST SPRING, Grant Number JPMJSP2132.
}

\subjclass[2020]{
05E30 Association schemes, strongly regular graphs, 
20D60 Arithmetic and combinatorial problems,
20E18 Limits, profinite groups
}
\date{\today}

\begin{abstract}
A wreath product is a method to construct
an association scheme from two association schemes. 
We determine the automorphism group of a wreath product.
We show a known result that a wreath product
is Schurian if and only if both components are
Schurian, which yields large families of 
non-Schurian association schemes and 
non-Schurian $S$-rings.
We also study iterated wreath products. 
Kernel schemes by Martin and Stinson are shown
to be iterated wreath products of class-one
association schemes. 
The iterated wreath products give
examples of projective systems of non-Schurian association schemes,
with an explicit description of primitive idempotents.
\end{abstract}

\maketitle
\section{Introduction}
Association schemes are central objects in algebraic combinatorics,
with many interactions with other areas of mathematics.
The wreath product is a method to 
construct an association scheme from two association 
schemes. 
This notion appears
in the monograph \cite[P.45-47]{WEIS} by Weisfeiler,
in a more general context of coherent configurations
(in the terminology of cellular algebras).
Song\cite{SONG} gives a description 
in terms of association schemes and their adjacency matrices. 
Muzychuk\cite{MUZYCHUK-WEDGE} defines the wreath product
of association schemes
and its generalizations (he used the term ``wedge product'').
There is a construction 
of non-Schurian $S$-rings by using a generalized 
wreath product for $S$-rings, studied by
Evdokimov and Ponomarenko\cite{EV-NONSCHUR}.

In this paper, we 
construct projective systems of 
non-Schurian association schemes in 
Section~\ref{sec:profin}.
As preliminaries, we recall some basic facts
on wreath products in Section~\ref{sec:wreath}
with a self-contained proof.
In Section~\ref{sec:auto},
we shall study the automorphism group of 
wreath products and their Schurian
property.
A large part of the results
in Section~\ref{sec:const-non-sch} is covered by more
detailed discussions in the monograph by Chen and Ponomarenko
\cite[Sections~3.2 and 3.4]{CHEN-PON},
but we give proofs for self-containedness, e.g., on
the known fact that 
the wreath product is Schurian if and only if both components
are Schurian 
(\cite[Corollary~3.4.7]{CHEN-PON}).
This yields a large 
class of non-Schurian association
schemes and non-Schurian $S$-rings, from the known examples.

In Section~\ref{sec:profin},
we consider iterated wreath products.
The kernel scheme by
Martin-Stinson\cite{MARTIN-STINSON} is an example.
An infinite iterated wreath product gives a projective
system of association schemes,
namely, profinite association schemes in the sense
of \cite{MOO},
with an explicit description of their primitive idempotents.
In \cite{MOO}, only Schurian examples are given. 
One of our motivations is to give examples of projective 
systems of non-Schurian association schemes.
\section{Wreath products}
\label{sec:wreath}
\subsection{Category of association schemes}
Let us recall the notion of association schemes briefly.
See Bannai-Ito\cite{Bannai} and Delsarte\cite{DELSARTE} for details.
We summarize basic terminologies.
\begin{definition}
Let $X$ be a finite set. By $\#X$ we denote the cardinality of
$X$. Let $C(X)$ denote the vector space of 
mappings from $X$ to $\C$. 
By the multiplication of functions, 
$C(X)$ is a unital commutative ring.
The set $C(X\times X)$
is naturally identified with the set of complex 
square matrices of size $\#(X)$, and the matrix product
is given by $AB(x,z)=\sum_{y\in X}A(x,y)B(y,z)$. The Hadamard
product $\circ$ is given by the component-wise product, namely,
$(A\circ B)(x,z)=A(x,z)B(x,z)$. (Note that $\circ$ may denote
the composition of mappings, but no confusion would occur.)
\end{definition}
\begin{definition}\label{def:assoc-scheme}
Let $X$, $I$ be finite sets, and $R:X\times X \to I$ a surjection.
We call $(X,R,I)$ an association scheme, 
if the following properties (1), (2), and (3) are satisfied.
For each $i \in I$, 
$R^{-1}(i)$ may be regarded as a relation on $X$, denoted by $R_i$.
Let $A_i$ be the corresponding adjacency matrix
in $C(X\times X)$. 
The surjection $R$ induces 
an injection $C(I) \to C(X\times X)$.
Let $A_X\subset C(X\times X)$ be the image.
\begin{enumerate}
 \item There is an $i_0 \in I$ such that $A_{i_0}$
is an identity matrix.
 \item $A_X\subset C(X\times X)$ is closed under the matrix product.
 \item $A_X$ is closed under the transpose of $C(X\times X)$.
\end{enumerate}
The algebra $A_X$ with the two multiplications
(i.e.\ the Hadamard product and the matrix product) is called
the Bose-Mesner algebra of $(X,R,I)$.
The set of $A_i$ $(i\in I)$ is the set of primitive idempotents
with respect to Hadamard products, and each is 
called an Hadamard primitive idempotent (or an adjacency matrix).
We call $X$ the underlying set of the association scheme, and
$I$ the set of the relations. The number $\#I-1$ is called
the number of classes.
We may use the same notation
$i_0$ for distinct association schemes.
The number of $1$ in a row of $A_i$ is independent 
of the choice of the row, called 
the $i$-th valency and denoted by $k_i$. It is also 
the number of $1$ in each column of $A_i$.
If $A_X$ is commutative with respect to the matrix product, 
then $(X,R,I)$ is said to be commutative.
\end{definition}
In the following, we write simply an ``association scheme $X$''
for an association scheme $(X,R,I)$ by an abuse of language.
\begin{definition}
For a commutative association scheme $X$, 
it is known that $A_X$ with matrix product is
isomorphic to the direct product of $\#I$ copies of $\C$
as a ring. Elements corresponding to $(0,\ldots,0,1,0,\ldots,0)$
are called the primitive idempotents.
We denote by $J=J(X)$
the set of primitive idempotents. Let $E_j$ denote
the primitive idempotent corresponding to $j\in J$.
The notation $j_0$ is kept for $E_{j_0}=\frac{1}{\#X}\bJ$,
where $\bJ$ denotes the matrix whose components are all $1$.
We use the symbol $j_0$ for any commutative association schemes.
\end{definition}
We follow MacLane \cite{MACLANE} for the terminologies of the category theory,
in particular, isomorphisms, functors, projective systems, and limits.
The association schemes form a category, 
by the following (e.g.\ Hanaki\cite{HANAKI-CAT} and Zieschang\cite{ZIESCHANG}).
\begin{definition}\label{def:cat}
Let $(X,R,I)$ and $(X',R',I')$ be association schemes.
A morphism of association schemes from $X$ to $X'$ is 
a pair of functions $f:X \to X'$ and $\sigma:I \to I'$ 
such that the following diagram commutes:
\begin{equation}\label{eq:comm}
    \begin{aligned}
      \begin{tikzpicture}[auto]
      \node (a1) at (0,2) {$X\times X$};
      \node (a2) at (0,0) {$X'\times X'$};
      \node (b1) at (3,2) {$I$};
      \node (b2) at (3,0) {$I'$};
      \node (1) at (1.5,1) {$\circlearrowleft$};
      \draw[->](a1) to node[swap] {$f\times f$} (a2);
      \draw[->](a1) to node {} (b1);
      \draw[->](a2) to node {} (b2);
      \draw[->](b1) to node {$\sigma$} (b2);
      \end{tikzpicture}.
    \end{aligned}
\end{equation}

By the commutativity, $\sigma$ preserves $i_0$.
It is clear that the surjectivity of $f$ implies 
that of $\sigma$. The morphism $(f,\sigma)$ is said to be 
surjective if $f$ is surjective. 
\end{definition}
\subsection{Wreath products}
Although originally the wreath product is introduced
by Weisfeiler \cite[P.45]{WEIS},
this section follows Song\cite[\S4]{SONG}, since
his description is due to association schemes
and convenient for our purpose. 
The definition, a proof, and the eigenmatrix
of wreath products
are all given there. We recall these, 
partly because we use some different symbols
and notations which make the definition
and proof simpler, and partly because we use some part of the 
proof in the following arguments.
\begin{definition}
Let $(X,R_X,I_X)$ and $(Y,R_Y,I_Y)$ be association schemes.
Then, 
$$
R_X \times R_Y: 
(X \times Y) \times (X \times Y) \to I_X \times I_Y
$$
is an association scheme, called the direct product of
$X$ and $Y$ \cite[\S3]{SONG} (called the tensor product
in \cite{WEIS}). 
In fact, this is the direct product
in the category of association schemes.
\end{definition}
It is easy to see the following.
\begin{proposition}\label{prop:direct}
The adjacency matrices 
of the direct product are
$$
\{A_{i_X}\otimes A_{i_Y} \mid i_X \in I_X, i_Y \in I_Y\}.
$$
If both $X$ and $Y$ are commutative, 
then so is $X\times Y$, 
and the primitive idempotents are
$$
\{E_{j_X}\otimes E_{j_Y} \mid j_X \in J_X, j_Y \in J_Y\}.
$$
\end{proposition}

The wreath product is defined as follows.
The non-standard symbol $\lexprod$ comes from 
the ``L'' of lexicographic.
\begin{definition}\label{def:lexI}
Define
$$
\lex: I_X \times I_Y \to (I_X\setminus \{i_0\})\coprod I_Y
$$ 
by
$$
\lex(i_X, i_Y)=
\begin{cases}
i_X & \mbox{ if } i_X \neq i_0, \\
i_Y & \mbox{ if } i_X = i_0.
\end{cases}
$$
We denote
$$
I_X \lexprod I_Y:=(I_X\setminus \{i_0\})\coprod I_Y.
$$
\end{definition}
\begin{definition}\label{def:functor}
For a mapping $\sigma:I_X \to I_X$ preserving $i_0\in I_X$
and a mapping $\tau:I_Y \to I_Y$ preserving $i_0\in I_Y$,
we define
$$
\sigma\lexprod \tau: I_X\lexprod I_Y \to I_X \lexprod I_Y
$$
by $\sigma$ on $I_X\setminus \{i_0\}$ 
and by $\tau$ on $I_Y$. (See Definition~\ref{def:cat} for
the preservation of $i_0$.)
We also define
$$
\pi_{I_X}:I_X \lexprod I_Y \to I_X, \quad i_X \mapsto i_X
\mbox{ and } i_Y \mapsto i_0.
$$
\end{definition}
\begin{remark}
This remark is for category-oriented readers,
and may be skipped since it is not essential in this paper.
Let $\Sets^!$ be the category of sets with a base point,
namely, an object is a set
with a base point and a morphism is a mapping preserving
the base points. In other words, for a fixed singleton
$\{!\}$, $\Sets^!$ is the category of 
morphisms $\{!\} \to S$
with $S$ being a set (a special case of comma categories).
We remark that $I_X \lexprod I_Y$ gives the coproduct functor
$\Sets^! \times \Sets^! \to \Sets^!$ 
(i.e. the pushout of $\{!\}\to I_X$ and $\{!\}\to I_Y$
in the category $\Sets$ of sets).
The first definition in 
Definition~\ref{def:functor}
comes from this functoriality. 
Remark that $\lex$ is NOT a natural transformation 
$$\lex: \times \Rightarrow \lexprod$$
where $\times:\Sets^! \times \Sets^! \to \Sets^!$ is the
direct product functor. On the other hand, there is a canonical
natural transformation 
$$\lexprod \Rightarrow \times$$
given by 
$$
I_X\lexprod I_Y \to I_X \times I_Y, \quad
i_x\in I_X \mapsto (i_x,i_0), \mbox{ and }
i_y\in I_Y \mapsto (i_0,i_y).
$$
Consequently, if
$$\pr_i:\Sets^! \times \Sets^! \to \Sets^!$$
is the projection functor to the $i$-th component $(i=1,2)$,
then there is a natural transformation $\lexprod \Rightarrow \pr_i$
for each $i=1,2$,
and $\pi_{I_X}$ is the evaluation of $\lexprod \Rightarrow \pr_1$
at $(I_X,I_Y)$.
\end{remark}
\begin{definition}\label{def:lex} (Wreath products)

Let $(X,R_X,I_X)$ and $(Y,R_Y,I_Y)$ be association schemes. 
Then the composition
$$
R_\lexprod:(X \times Y) \times (X \times Y) \to I_X \times I_Y 
\stackrel{\lex}{\to} I_X \lexprod I_Y
$$
is an association scheme, which is called the 
wreath product of $X$ and $Y$, and denoted by 
$$
X \lexprod Y.
$$
\end{definition}
\begin{remark}
The underlying set of $X\lexprod Y$ is $X\times Y$, 
whereas the set of relations is $I_X\lexprod I_Y$ with 
the cardinality $\#I_X +\#I_Y -1$. The wreath product
is a fusion of the direct product.
\end{remark}
We first thought that the name ``wreath product''
seems misleading, since this seems to have no relation with
the wreath product of groups (as Song \cite{SONG} remarked in 
his introduction). However,
there is a strong relation \cite[P.47, \S4.8]{WEIS}
via transitive permutation groups.
Still, we point out the similarity to 
the ``lexicographic ordering,''
since the above definition means that the relation
of two elements in $X \times Y$
is determined by first looking at the $X$-components,
and decided according to their relation 
when they are different, 
and otherwise by the relation at the $Y$-components.
This procedure is similar to the lexicographic ordering.
The next is a translation of \cite[Theorem~4.1]{SONG}.
\begin{theorem}\label{th:main} $ $
\begin{enumerate}
\item 
For association schemes $X$ and $Y$, 
their wreath product $X\lexprod Y$ is an association scheme.
\item 
Let $A_{i_X}$ $(i_X \in I_X)$ 
(and $A_{i_Y}$ $(i_Y\in I_Y)$, respectively) 
be the adjacency matrices of 
$X$ (those of $Y$, respectively). Then,
the adjacency matrices of $X \lexprod Y$ are 
as follows.
\[
\begin{array}{lcl}
\mbox{ Type front: }& A_{i_X}\otimes \bJ & \mbox{ if } i_X\in I_X, i_X \neq i_0 \\
\mbox{ Type rear: } & \bI \otimes A_{i_Y}& \mbox{ if } i_Y\in I_Y, \\
\end{array}
\]
where $\bI$ denotes the identity matrix (of size $\#X$)
and $\bJ$ denotes the matrix whose components are all $1$
(square of size $\#Y$).
\end{enumerate}
\end{theorem}
\begin{proof}
Take an element $i_X \in I_X\setminus \{i_0\}\subset I_X \lexprod I_Y$. 
Its preimage by $\lex$ is $\{(i_X, i_Y) \mid i_Y \in I_Y\}$.
By definition, 
$$
\sum_{i_Y\in I_Y} A_{i_Y}= \bJ.
$$
By Proposition~\ref{prop:direct}, the inverse image
of $i_X \in I_X \setminus\{i_0\}\subset I_X \lexprod I_Y$ by $R_\lexprod$ corresponds to 
$$
A_{i_X} \otimes \bJ.
$$
Take an element $i_Y \in I_Y \subset I_X \lexprod I_Y$. 
Its inverse image by $R_\lexprod$ corresponds to 
$$
\bI \otimes A_{i_Y}.
$$
To check that these constitute an association scheme,
we only need to show that this set of matrices is
closed under the transpose and contains the unit, which are
obvious, and that
matrix products of these matrices
are linear combinations of these. A product of the front-type
matrices is a linear combination of those, since $X$ is an association
scheme. A product of the rear-type matrices is a linear combination
of those, since $Y$ is an association scheme.
For the mixed case, by $\bJ A_{i_Y}=k_{i_Y}\bJ$, 
the product is a scalar
multiple of a matrix of the front-type.
\end{proof}
Here we start our own observations.
\begin{corollary}\label{cor:non-P}
The wreath product $X \lexprod Y$ is not $P$-polynomial
if $\#X>1$ and $\#Y>1$. 
\end{corollary}
\begin{proof}
The wreath product is $P$-polynomial, if there is an
adjacency matrix that generates the Bose-Mesner
algebra as a ring over $\C$.
By $\#X>1$, we have an $i_X\neq i_0$. 
By $\#Y>1$, we have an $i_Y\neq i_0$.
The linear span of the front-type matrices 
is identical with the set $A_X\otimes \bJ$ and
is closed under the matrix multiplication.
The linear span of the rear-type matrices is 
identical with the set
$\bI\otimes A_Y$ and
is closed under the matrix multiplication.
Thus, a front-type adjacency matrix does not generate
a rear-type adjacency matrix $\bI\otimes A_{i_Y}$,
since $i_Y\neq i_0$ and hence $A_{i_Y}\neq \bJ$.
A rear-type adjacency matrix does not generate
a front-type adjacency matrix $A_{i_X}\otimes \bJ$
since $A_{i_X}\neq A_{i_0}=\bI$.
\end{proof}
\begin{corollary}
The wreath product of association schemes
is a commutative association scheme if and only if
both components are commutative. 
\end{corollary}
\begin{proof}
This follows from the above description of the Hadamard primitive idempotents. 
\end{proof}

\begin{corollary}\label{cor:lex-J}
Suppose that both $X$ and $Y$ are commutative.
Let $E_{j_X}$ $(j_X \in J_X)$ 
(and $E_{j_Y}$ $(j_Y\in J_Y)$, respectively) 
be the primitive idempotents of 
$X$ (of $Y$, respectively). Then,
the primitive idempotents of $X \lexprod Y$ are
as follows.
\[
\begin{array}{lcl}
\mbox{ Type front: }& E_{j_X}\otimes \frac{1}{\#Y}\bJ & \mbox{ if } j_X\in J_X\\
\mbox{ Type rear: } & \bI \otimes E_{j_Y}& \mbox{ if } j_Y\in J_Y, j_Y \neq j_0. \\
\end{array}
\]
Thus, 
$$
J(X\lexprod Y)=J_X \coprod (J_Y\setminus \{j_0\}).
$$
\end{corollary}
\begin{proof}
One sees that these elements are linear combinations
of the adjacency matrices in Theorem~\ref{th:main}
because $X$ and $Y$ are commutative association schemes,
and hence in the Bose-Mesner algebra of $X\lexprod Y$.
It is easy to check that these are idempotents
and the product of any two distinct elements is zero, and hence
are linearly independent.
Since the number of these matrices is $\#X+\#Y-1$,
which is the cardinality of $I_X\lexprod I_Y$, these
idempotents span the Bose-Mesner algebra.
\end{proof}
\begin{corollary}
Let $X$ and $Y$ be commutative association schemes. 
Then, $X\lexprod Y$ is not $Q$-polynomial, 
if $\#X>1$ and $\#Y>1$.
\end{corollary}
\begin{proof} A similar proof to that of Corollary~\ref{cor:non-P}
applies to.
\end{proof}

\begin{corollary} 
We use $E_{j_0}=\frac{1}{\#Y}\bJ$.
The first eigenmatrix of $X\lexprod Y$ is given by
\begin{eqnarray*}
(A_{i_X}\otimes \bJ)(E_{j_X}\otimes E_{j_0})
&=&
P_{i_X}(j_X)\#Y (E_{j_X}\otimes E_{j_0}), \\
(\bI \otimes A_{i_Y})(E_{j_X}\otimes E_{j_0})
&=&
k_{i_Y}(E_{j_X}\otimes E_{j_0}), \\
(A_{i_X}\otimes \bJ)(\bI \otimes E_{j_Y})
&=&
0 (\bI \otimes E_{j_Y}), \\
(\bI \otimes A_{i_Y})(\bI \otimes E_{j_Y})
&=&
P_{i_Y}(j_Y)(\bI \otimes E_{j_Y}), 
\end{eqnarray*}
where $P_{i_X}(j_X)$ 
is the eigenvalue of $A_{i_X}$ for eigenvector $E_{j_X}$, and
$P_{i_Y}(j_Y)$ 
is the eigenvalue of $A_{i_Y}$ for eigenvector $E_{j_Y}$.
\end{corollary}

\begin{proposition}\label{prop:proj}
Let $X$ and $Y$ be association schemes.
The projection 
$$
p:X\times Y \to X
$$
of the underlying sets and the mapping
$$
\pi_{I_X}:I_X \lexprod I_Y \to I_X
$$
in Definition~\ref{def:functor} give
a surjective morphism of association schemes
$$
\pi:X \lexprod Y \to X.
$$
\end{proposition}
\begin{proof}
This follows by a diagram chasing.
\end{proof}

\section{Automorphisms and Schurian property}
\label{sec:auto}
\subsection{Automorphism group of a wreath product}
\begin{definition}\label{def:aut-sets}
For an association scheme $X$, we denote by $\Aut(X)$
the group of automorphisms in the category of
association schemes (Definition~\ref{def:cat}).
On the other hand, for any set $X$ equipped with any structure
(such as an association scheme or a group),
$$
\Aut_\Sets(X)
$$
denote the group of bijections from $X$ to $X$,
neglecting the structure of $X$.
\end{definition}

As stated before, the wreath product of association schemes is
studied and generalized by 
Muzychuk\cite{MUZYCHUK-WEDGE},
where in Proposition~3.2 the automorphism group of
generalized wreath products of thin-association schemes
is determined. 
Chen and Ponomarenko \cite[Theorem~3.4.6]{CHEN-PON}
determined the automorphism group of
the wreath product, but their definition 
of automorphism groups is the normal subgroup
$\Aut(X|I) < \Aut(X)$ 
in Lemma~\ref{lem:X-I} below, 
and thus different from ours.
\footnote{Our definition of $\Aut$ coincides with $\Iso$
in \cite[Definition~2.2.1]{CHEN-PON}, which Chen and Ponomarenko
say ``definitely more natural than $\Aut$, 
but here we follow a long tradition.''}
We give a description of the automorphism group 
of wreath products of association schemes.

\begin{theorem}\label{th:aut}
Let $X$ and $Y$ be association schemes,
and $p:X \times Y \to X$ be the projection.
Then, for any $(f,\sigma) \in \Aut(X\lexprod Y)$,
$f$ maps each fiber $Y_x:=p^{-1}(x)$ to another fiber.
Thus, $f$ induces an element $\pi(f)\in \Aut_\Sets(X)$.
This gives a surjective morphism 
of groups
$$
\Aut(X\lexprod Y) \stackrel{\gpr}{\to} \Aut(X)
$$ 
with a natural splitting morphism.
Thus, we have a group isomorphism
$$
\Aut(X\lexprod Y) \cong K \rtimes \Aut(X),
$$
where $K:=\ker P$, whose structure is
described in Proposition~\ref{prop:ker} below.
We have a natural embedding
of $\Aut(Y)$ into $K$ that acts trivially on $X$
and diagonally on each fiber $Y_x$.
The image of the embedding commutes with the image of the splitting.
Thus, $\Aut(X\lexprod Y)$
has a subgroup isomorphic to 
$\Aut(Y)\times \Aut(X)$.
\end{theorem}
\begin{proof}
We construct $P$. 
An automorphism in $\Aut(X\lexprod Y)$ is a pair of bijections
$$f:X \times Y \to X \times Y,$$ 
$$\tau : I_X\lexprod I_Y \to I_X\lexprod I_Y$$
which makes the diagram 
\begin{equation}\label{eq:tau}
    \begin{aligned}
      \begin{tikzpicture}[auto]
      \node (a1) at (0,2) {$(X\times Y)\times(X\times Y)$};
      \node (a2) at (0,0) {$I_{X}\lexprod I_{Y}$};
      \node (b1) at (4.5,2) {$(X\times Y)\times(X\times Y)$};
      \node (b2) at (4.5,0) {$I_{X}\lexprod I_{Y}$};
      \node (1) at (2.25,1) {$\circlearrowleft$};
      \draw[->](a1) to node[swap] {$R_{\lexprod}$} (a2);
      \draw[->](a1) to node {$f\times f$} (b1);
      \draw[->](a2) to node {$\tau$} (b2);
      \draw[->](b1) to node {$R_{\lexprod}$} (b2);
      \end{tikzpicture}
    \end{aligned}
\end{equation}
commute.
We claim that $\tau$ permutes $I_X\setminus \{i_0\}$ and $I_Y$ separately.
Take an element $i_X \in I_X\setminus \{i_0\}$. 
Then the corresponding valency is $k_{i_X}\#Y\geq \#Y$, by 
Theorem~\ref{th:main}. For $i_Y\in I_Y$, the valency is
$k_{i_Y}\leq\#Y$, with equality holds only if 
$I_Y=\{i_0\}$. Thus, there is no automorphism that maps $i_X$ to $i_Y$.
Thus we have $\pi(\tau)$ which makes the following commute:
\begin{equation}\label{eq:I}
    \begin{aligned}
      \begin{tikzpicture}[auto]
      \node (a1) at (0,2) {$I_{X}\lexprod I_{Y}$};
      \node (a2) at (0,0) {$I_{X}$};
      \node (b1) at (3,2) {$I_{X}\lexprod I_{Y}$};
      \node (b2) at (3,0) {$I_{X}$};
      \node (1) at (1.5,1) {$\circlearrowleft$};
      \draw[->](a1) to node[swap] {$\pi_{I}$} (a2);
      \draw[->](a1) to node {$\tau$} (b1);
      \draw[->](a2) to node {$\pi(\tau)$} (b2);
      \draw[->](b1) to node {$\pi_{I}$} (b2);
      \end{tikzpicture}.
    \end{aligned}
\end{equation}
For $(x,y)$ and $(x,y')$, since 
$$
R_\lexprod((x,y),(x,y'))\in I_Y \subset I_X\lexprod I_Y,
$$ 
we have
$$
R_\lexprod(f(x,y),f(x,y'))\in I_Y \subset I_X\lexprod I_Y.
$$
This means that
the $X$-component of $f(x,y)$ is the same with 
that of $f(x,y')$, i.e., $f$ maps a fiber of $p$
to another fiber, which shows the unique existence of 
$\pi(f)$ that makes a commutative diagram
\begin{equation}\label{eq:comm-f}
    \begin{aligned}
      \begin{tikzpicture}[auto]
      \node (a1) at (0,2) {$X\times Y$};
      \node (a2) at (0,0) {$X$};
      \node (b1) at (3,2) {$X\times Y$};
      \node (b2) at (3,0) {$X$.};
      \node (1) at (1.5,1) {$\circlearrowleft$};
      \draw[->](a1) to node[swap] {$p$} (a2);
      \draw[->](a1) to node {$f$} (b1);
      \draw[->](a2) to node {$\pi(f)$} (b2);
      \draw[->](b1) to node {$p$} (b2);
      \end{tikzpicture}
    \end{aligned}
\end{equation}
Now we have
\begin{equation}\label{eq:hexagon}
    \begin{aligned}
      \begin{tikzpicture}[auto]
      \node (a) at (0,0) {$X\times X$};
      \node (b1) at (3,2) {$(X\times Y)\times(X\times Y)$};
      \node (b2) at (3,0) {$I_{X}\lexprod I_{Y}$};
      \node (b3) at (3,-2) {$I_{X}$};
      \node (c1) at (10.2-3,2) {$(X\times Y)\times(X\times Y)$};
      \node (c2) at (10.2-3,0) {$I_{X}\lexprod I_{Y}$};
      \node (c3) at (10.2-3,-2) {$I_{X}$,};
      \node (d) at (10.2,0) {$X\times X$};
      \node (1) at (1.6,0) {$\circlearrowleft$};
      \node (2) at (5.1,1) {$\circlearrowleft$};
      \node (3) at (5.1,-1) {$\circlearrowleft$};
      \node (4) at (10.2-1.6,0) {$\circlearrowleft$};
      \draw[->](b1) to node {} (a);
      \draw[->](a) to node {} (b3);
      \draw[->](b1) to node {} (b2);
      \draw[->](b2) to node {} (b3);
      \draw[->](b1) to node {$f\times f$} (c1);
      \draw[->](b2) to node {$\tau$} (c2);
      \draw[->](b3) to node {$\pi(\tau)$} (c3);
      \draw[->](c1) to node {} (c2);
      \draw[->](c2) to node {} (c3);
      \draw[->](c1) to node {} (d);
      \draw[->](d) to node {} (c3);
      \end{tikzpicture}
    \end{aligned}
\end{equation}
where the commutativity of the two triangles
follows from Proposition~\ref{prop:proj}.
This commutativity, (\ref{eq:comm-f}), and the surjectivity of the projection 
$X\times Y \to X$ conclude the commutativity of 
\begin{equation}\label{eq:comm-tau}
    \begin{aligned}
      \begin{tikzpicture}[auto]
      \node (a1) at (0,2) {$X\times X$};
      \node (a2) at (0,0) {$I_{X}$};
      \node (b1) at (3.5,2) {$X\times X$};
      \node (b2) at (3.5,0) {$I_{X}$};
      \node (1) at (1.75,1) {$\circlearrowleft$};
      \draw[->](a1) to node[swap] {} (a2);
      \draw[->](a1) to node {$\pi(f)\times\pi(f)$} (b1);
      \draw[->](a2) to node {$\pi(\tau)$} (b2);
      \draw[->](b1) to node  {} (b2);
      \end{tikzpicture},
    \end{aligned}
\end{equation}
and hence gives an automorphism of the association scheme $X$,
which gives the group homomorphism $\gpr$.
Conversely, an automorphism $(g,\tau)$ of $X$,
namely, a pair of bijections $g:X \to X$, $\tau:I_X\to I_X$ with
an additional commutativity condition
gives a pair of bijections $g \times \id_Y:X\times Y \to X\times Y$
and $\tau \lexprod \id_{I_Y}: I_X \lexprod I_Y \to I_X \lexprod I_Y$
(see Definition~\ref{def:functor}), 
which is an automorphism 
of the association scheme $X\lexprod Y$, splitting the homomorphism $\gpr$.
On the other hand, it is easy to show 
that any automorphism $(g,\sigma)$ of $Y$ induces
an automorphism $(f,\tau)$ of $X\lexprod Y$, by putting $f=\id_X\times g$
and $\tau = \id_{I_X} \lexprod \sigma$ (see Definition~\ref{def:functor}),
which lies in the kernel of $\gpr$ and commutes with 
the image of the splitting of $\gpr$.
\end{proof}

On the kernel of $\gpr$, we show the following.
\begin{proposition}~\label{prop:ker}
The kernel $K$ of $\gpr:\Aut(X\lexprod Y) \to \Aut(X)$
is a subgroup of $\prod_{x\in X}\Aut(Y_x)$. 
More precisely,
we have 
\begin{equation}\label{eq:K}
K = \coprod_{\tau \in T}(\prod_{x\in X} \Aut(Y_x)_\tau), 
\end{equation}
where $T$ is the image in $\Aut_\Sets(I_Y)$ of $\Aut(Y)$,
and $\Aut(Y_x)_\tau$ denotes the set of elements in $\Aut(Y_x)$
whose $\Aut_\Sets(I_Y)$-component is $\tau$. (Note that
for $(f,\tau)\in \Aut(Y)$, $f$ uniquely determines $\tau$,
hence the above union is disjoint.)

Thus, there is an injective group homomorphism
$$
\Aut(X\lexprod Y) \hookrightarrow \Aut(Y_x) \wr_X \Aut(X),
$$
where the right-hand side is the wreath product of 
groups
with respect to the action of $\Aut(X)$
on $X$.
\end{proposition}
\begin{proof}
We compute the kernel $K$ of $P$. 
It is clear that
$$
K \subset \prod_{x \in X} \Aut_\Sets(Y_x).
$$
Take any $(k,\sigma)\in K$ and $x\in X$, and look the action on $Y_x$.
Clearly $\sigma=(\id_{I_X}\lexprod (\sigma|_{I_Y}))$.
There exists $k_x\in \Aut_\Sets(Y)$
such that $k(x,y)=(x,k_x(y))$.
For any $y_1, y_2\in Y$, 
\begin{eqnarray*}
R_Y(k_x(y_1),k_x(y_2))
&=&R_\lexprod((x,k_x(y_1)),(x,k_x(y_2))) \\
&=&R_\lexprod(k(x,y_1),k(x,y_2))) \\
&=&\sigma \circ R_\lexprod((x,y_1),(x,y_2))) \\
&=&\sigma|_{I_Y} \circ R_Y(y_1,y_2) \\
\end{eqnarray*}
implies that $(k_x, \sigma|_{I_Y}) \in \Aut(Y)$.
This holds for any $x\in X$, and thus 
$$
(k,\id_{I_X}\lexprod (\sigma|_{I_Y}))\in \prod_{x\in X}\Aut(Y_x)_{\sigma|_{I_Y}}.
$$
It follows that
$$
K \subset \coprod_{\tau \in T}\prod_{x \in X} \Aut(Y_x)_\tau.
$$
Conversely, take an element from the right-hand side of (\ref{eq:K}):
$$
 k=\prod_{x\in X}k_x \in \prod_{x\in X}\Aut_\Sets(Y_x)\subset
 \Aut_\Sets(X\times Y). 
$$
Thus, $k_x$ shares $\tau \in \Aut_\Sets(I_Y)$ for any $x$
such that $(k_x,\tau)\in \Aut(Y)$.
We claim that $(k, \id_{I_X}\lexprod \tau) \in K$.
In fact, for $(x_1,y_1)$ and $(x_2,y_2)$,
if $x_1\neq x_2$, then
\begin{eqnarray*}
R_\lexprod(k(x_1,y_1),k(x_2,y_2))
&=&
R_\lexprod((x_1,k_{x_1}(y_1)),(x_2,k_{x_2}(y_2)))\\
&=&
R_X(x_1,x_2) \\
&=&
(\id_{I_X}\lexprod \tau) \circ R_\lexprod((x_1,y_1),(x_2,y_2)),
\end{eqnarray*}
and if 
$x_1 = x_2$, then
\begin{eqnarray*}
R_\lexprod(k(x_1,y_1),k(x_2,y_2))
&=&
R_\lexprod(x_1,k_{x_1}(y_1)),(x_2,k_{x_2}(y_2))\\
&=&
R_Y(k_{x_1}(y_1), k_{x_2}(y_2)) \\
&=&
R_Y(k_{x_1}(y_1), k_{x_1}(y_2)) \\
&=&
\tau\circ R_Y(y_1, y_2) \\
&=&
(\id_{I_X}\lexprod \tau) \circ R_\lexprod((x_1,y_1),(x_2,y_2)),
\end{eqnarray*}
which imply $(k, \id_{I_X}\lexprod \tau) \in K$.
Thus 
$$
K = \coprod_{\tau \in T}\prod_{x \in X} \Aut(Y_x)_\tau
$$
follows. The rest of the claims, i.e., the relation 
to the wreath product, holds at the level of 
permutation groups. That is, the set of permutations of
$X \times Y$ that permute $Y_x$ block-wise is isomorphic to
$$
\Aut_\Sets(Y_x)\wr_X \Aut_\Sets(X).
$$
The detail is omitted.
\end{proof}
\commentout{
The stabilizer group of 
$\Aut(X\lexprod Y)$ of $x\in X$
acts on the fiber $Y_x$ through $\Aut(Y)$.
\begin{proposition}
Let $X$ and $Y$ be association schemes.
Let $p:X\times Y \to X$ be the projection, 
and for $x\in X$, $Y_x$ denote the fiber $p^{-1}(x)$,
naturally identified with $Y$. By Theorem~\ref{th:aut},
$\Aut(X\lexprod Y)$ acts on the underlying set $X$.
Let $G_x\subset \Aut(X\lexprod Y)$ be the stabilizer of $x\in X$.
Then, $G_x$ acts on $Y_x$, which gives a 
group homomorphism
$$
G_x \to \Aut(Y),
$$
which is surjective and splitting.
The splitting is given in Theorem~\ref{th:aut}.
\end{proposition}
\begin{proof}
Clearly $G_x$ acts on $Y_x=Y$. 
For $(f,\sigma)\in G_x$,
we may define $f'\in \Aut_\Sets(Y)$
such that $f(x,y)=(x,f'(y))$.
If
$$
i_Y=R_\lexprod((x,y),(x,y'))=R_Y(y,y'),
$$ 
then
$$
\sigma(i_Y)=R_\lexprod(f(x,y),f(x,y'))
=R_\lexprod((x,f'(y)),(x,f'(y')))
=R_Y(f'(y), f'(y')).
$$
This gives $(f',\sigma|_{I_Y})\in \Aut(Y)$,
and hence a group homomorphism
$$
G_x \to \Aut(Y).
$$
This has a splitting, since the image of the
embedding 
$$
\Aut(Y) \to \Aut(X\lexprod Y)
$$
in Theorem~\ref{th:aut} is in $G_x$, and restriction
to the fiber $Y_x$ gives the identity.
\end{proof}
}

\subsection{Construction of non-Schurian schemes}
\label{sec:const-non-sch}
The results of this subsection are known to 
the researchers in this area, and proofs are given
in Chen and Ponomarenko 
\cite[Sections~3.2 and 3.4]{CHEN-PON}
with more systematic and detailed studies. 
However, we felt that it might not easy
to follow the details, and decided to include our proofs,
for the reader's convenience. 
\begin{definition} (Schurian association schemes)
\label{def:schurian-sch}

Let $G$ be a group and $X$ a set with left transitive
action of $G$. Then
the quotient by the diagonal action
$$
 R:X \times X \to G\backslash (X\times X)=:I
$$
is known to be a (possibly non-commutative)
association scheme.
An association scheme isomorphic to 
this type is called a Schurian association scheme.
\end{definition}
We may replace $G$ with its image in $\Aut_\Sets(X)$.
Then, a pair $(g, \id_I)$ for $g \in G$ is an automorphism
of the Schurian scheme, and $G$ transitively acts on $R^{-1}(i)$
for each $i\in I$.
The following is immediate from this observation.
\begin{lemma}\label{lem:X-I}
Let $(X,R,I)$ be an association scheme.
Let $\Aut(X|I)\subset \Aut(X)$ be the subgroup
consisting of elements that
act on $I$ trivially. Then, 
$X$ is Schurian if and only if $\Aut(X|I)$ acts transitively on $R^{-1}(i)$
for each $i\in I$.
\end{lemma}

The next is \cite[Corollary~3.2.22]{CHEN-PON}.
\begin{proposition}\label{prop:schurian}
Let $X$ and $Y$ be association schemes.
The following are equivalent.
\begin{enumerate}
\item \label{enum:both}
Both $X$ and $Y$ are Schurian.
\item \label{enum:dir-prod}
The direct product $X\times Y$
is Schurian. 
\end{enumerate}
\end{proposition}
\begin{proof}
Assume (\ref{enum:both}). Then, it is easy
to check that we have a natural group homomorphism
$$
\Aut(X|I_X)\times \Aut(Y|I_Y)
\to \Aut(X\times Y | I_X \times I_Y).
$$
For any $(i_X,i_Y)\in I_X \times I_Y$, 
the subset 
$$
R_{X\times Y}^{-1}(i_X, i_Y)\subset (X\times Y)\times (X\times Y)
$$
is identified with 
$$
R_X^{-1}(i_X)\times R_Y^{-1}(i_Y)
$$
under the identification 
$(X\times Y)\times (X\times Y)=(X\times X)\times (Y\times Y)$,
on which $\Aut(X|I_X)\times \Aut(Y|I_Y)$ acts transitively
by the assumption, and by Lemma~\ref{lem:X-I}, 
we conclude (\ref{enum:dir-prod}).

Conversely, assume (\ref{enum:dir-prod}). By symmetry,
it suffices to show that $X$ is Schurian.
We claim that there is a natural group morphism
$$
\Aut(X\times Y | I_X \times I_Y) \to \Aut(X|I_X).
$$
Let $p:X \times Y \to X$ be the projection.
Take $(f,\id) \in \Aut(X\times Y | I_X \times I_Y)$.
Take any $x\in X$ and $y, y'\in Y$.
Then 
$$
R_{X\times Y}(f(x,y),f(x,y'))=
R_{X\times Y}((x,y),(x,y'))=(i_0,R_Y(y,y'))
$$
holds. Thus $p(f(x,y))=p(f(x,y'))$, hence we have $\pi(f)$
that makes the diagram (\ref{eq:comm-f}) commute.
Take any $y\in Y$.
If we denote $I_X\times I_Y \to I_X$ by the same
symbol $p$, we have
\begin{eqnarray*}
R_X(\pi(f)(x),\pi(f)(x')) 
&=&
R_X(p\circ f(x,y),p\circ f(x',y)) \\
&=&
p\circ R_{X\times Y}(f(x,y),f(x',y)) \\
&=&
p\circ R_{X\times Y}((x,y),(x',y)) \\
&=&
R_{X\times Y}(p(x,y),p(x',y)) \\
&=& R_X(x,x'),
\end{eqnarray*}
which implies $(\pi(f),\id)\in \Aut(X|I_X)$.
Suppose that 
$$
R_X(x_1,x_1')=R_X(x_2,x_2')=i_X.
$$
Take any $y\in Y$. Then 
$$
R_{X\times Y}((x_1,y),(x_1',y))
=R_{X\times Y}((x_2,y),(x_2',y))=(i_X,i_0).
$$
Thus by the assumption (\ref{enum:dir-prod}), 
we have $(f,\id)\in \Aut(X\times Y | I_X \times I_Y)$
such that
$$
f(x_1,y)=(x_2,y), \quad f(x_1',y)=(x_2',y).
$$
Then, by taking the image of $p$, 
$$
\pi(f)(x_1)=x_2, \quad \pi(f)(x_1')=x_2'.
$$
Thus, $\Aut(X|I_X)$ transitively acts on 
$R_X^{-1}(i_X)$, which shows (\ref{enum:both}).
\end{proof}

The next is \cite[Corollary~3.4.7]{CHEN-PON}.
\begin{theorem}\label{th:schurian}
Let $X$ and $Y$ be association schemes.
The following are equivalent.
\begin{enumerate}
\item \label{enum:both-wr}
Both $X$ and $Y$ are Schurian.
\item \label{enum:wr-prod}
The wreath product $X\lexprod Y$
is Schurian. 
\end{enumerate}
\end{theorem}
\begin{proof}
Assume (\ref{enum:wr-prod}).
Put 
$G:=\Aut(X\lexprod Y|I_X\lexprod I_Y)$ as in Lemma~\ref{lem:X-I}.
Take any $(x_1,x_1')$, $(x_2,x_2')$ and $i_X \in I_X$ with 
$$
R_X(x_1,x_1')=R_X(x_2,x_2')=i_X.
$$
We take an arbitrary $y\in Y$,
then 
$$
R_\lexprod((x_1,y),(x_1',y))=R_\lexprod((x_2,y),(x_2',y))
$$
holds by the case division for $x_1=x_1'$ or not. 
We consider the image $\gpr(G)\subset \Aut(X)$ in Theorem~\ref{th:aut}.
Since $X\lexprod Y$ is Schurian,
there is an automorphism $(f,\tau)\in G$ such that
$\tau=\id_{I_X\lexprod I_Y}$,
$f(x_1,y)=(x_2,y)$ and $f(x_1',y)=(x_2',y)$. By Theorem~\ref{th:aut},
there is an automorphism $\gpr((f,\tau))=(\pi(f),\pi(\id))\in \gpr(G)$ of $X$, 
which maps $x_1$ to $x_2$ and $x_1'$ to $x_2'$,
hence $X$ is Schurian. For $Y$, take arbitrary 
$(y_1,y_1')$, $(y_2,y_2')$ and $i_Y\in I_Y$
with 
$$R_Y(y_1,y_1')=R_Y(y_2,y_2')=i_Y.$$
Fix an arbitrary $x \in X$, identify $Y=\{x\}\times Y \subset X\times Y$,
and let $G_x$ be the stabilizer of $\{x\}\times Y$ in $G$
(i.e. the set of elements in $G$ that preserve $\{x\}\times Y$ as a set). 
Hence $G_x$ 
acts on $Y$, and by definition of $G$, $G_x$
trivially acts on $I_Y$ and consequently
$G_x\subset \Aut(Y|I_Y)$ follows by the commutativity of
\begin{center}
    \begin{tikzpicture}[auto]
    \node (a1) at (0,2) {$(\{x\}\times Y)\times (\{x\}\times Y)$};
    \node (a2) at (0,0) {$(X\times Y)\times (X\times Y)$};
    \node (b1) at (4,2) {$\{i_0\}\times I_{Y}$};
    \node (b2) at (4,0) {$I_{X}\lexprod I_{Y}$};
    \draw[{Hooks[right]}->](a1) to node[swap] {} (a2);
    \draw[->](a1) to node {$R_X|_{\{x\}}\times R_{Y}$} (b1);
    \draw[->](a2) to node {$R_{\lexprod}$} (b2);
    \draw[{Hooks[right]}->](b1) to node {} (b2);
    \end{tikzpicture}.
\end{center}
Since
$$
R_\lexprod((x,y_1),(x,y_1'))=R_\lexprod((x,y_2),(x,y_2'))=i_Y
$$
and $X\lexprod Y$ is Schurian, 
there is an automorphism $(f,\id_{I_X\lexprod I_Y}) \in G$
with both $f(x,y_1)=(x,y_2)$ and $f(x,y_1')=(x,y_2')$ hold.
By (\ref{eq:comm-f}), this means $\pi(f)(x)=x$ and
$(f,\id)\in G_x$. Thus, $G_x\subset \Aut(Y|I_Y)$
acts transitively on $(R_X|_{\{x\}}\times R_Y)^{-1}((i_0,i_Y))$,
and hence $Y$ is Schurian. These imply (\ref{enum:both-wr}).

Assume (\ref{enum:both-wr}). Take four points in $X \times Y$
such that
\begin{equation}
R_\lambda((x_1,y_1), (x_1',y_1'))
=
R_\lambda((x_2,y_2), (x_2',y_2')).
\end{equation}
Assume that $x_1=x_1'$. Then, by the definition of $R_\lexprod$,
$x_2=x_2'$, and $R_Y(y_1,y_1')=R_Y(y_2,y_2')$.
Thus, there is an $(f,\id_X)\in \Aut(X|I_X)$
with 
$$
f\times f:(x_1,x_1') \mapsto (x_2,x_2')
$$ 
(since both lie in $R_X^{-1}(i_0)$).
There is a $(g,\id_Y) \in \Aut(Y|I_Y)$
with 
$$
g\times g:(y_1,y_1')=(y_2,y_2').
$$
By the last statement of Theorem~\ref{th:aut},
we have 
$$
(f,\id_X) \circ (g,\id_Y) \in \Aut(X\lexprod Y|I_X\lexprod I_Y)
$$
that maps 
$(x_1,y_1)\mapsto (x_2,y_2)$ and 
$(x_1',y_1')\mapsto (x_2',y_2')$, which implies (\ref{enum:wr-prod}).

Assume that $x_1 \neq x_1'$.
In this case, $R_X(x_1,x_1')=R_X(x_2,x_2')$ and $x_2\neq x_2'$, 
hence there is an $f\in \Aut(X|I_X)$ with
$$
f\times f:(x_1,x_1') \mapsto (x_2,x_2').
$$
Then $(f\times \id_Y|\id)\in \Aut(X\lexprod Y |I_X\lexprod I_Y)$
maps 
\begin{equation}\label{eq:maps-pair}
(x_1,y_1) \mapsto (x_2,y_1), \quad
(x_1',y_1') \mapsto (x_2',y_1').
\end{equation}
Recall that $\Aut(Y_{x_2}|I_{Y_{x_2}})=\Aut(Y_{x_2})_{\id_{I_Y}}$
in Proposition~\ref{prop:ker} lies in $K$,
and since $x_2 \neq x_2'$, 
$$
\Aut(Y_{x_2})_{\id_{I_Y}}
\times
\Aut(Y_{x_2'})_{\id_{I_Y}} \in K.
$$
Thus, since $Y$ is Schurian, there are 
a $g \in \Aut(Y_{x_2}|I_{Y_{x_2}})$ mapping
$y_1 \mapsto y_2$ and 
a $g'\in \Aut(Y_{x_2'}|I_{Y_{x_2}'})$ mapping
$y_1' \mapsto y_2'$. Thus, 
$g\circ g' \in \Aut(X\lexprod Y|I_X\lexprod I_Y)$
maps
\begin{equation}
(x_2,y_1)  \mapsto (x_2,y_2), \quad
(x_2',y_1')\mapsto (x_2',y_2'). 
\end{equation}
Together with (\ref{eq:maps-pair}), we conclude (\ref{enum:both-wr}).
\end{proof}

\begin{corollary}
The wreath product of a non-Schurian association 
scheme and an association scheme (in both order of 
product) is non-Schurian. The same statement holds
for the direct product.
\end{corollary}
Thus, there exists a large family of non-Schurian schemes. 
We remark that there are substantial studies on
construction of non-Schurian schemes, e.g., 
Evdokimov-Ponomarenko\cite{EV-NONSCHUR}, Hanaki-Hirai-Ponomarenko\cite{HANAKI-NONSCHUR}, and Hirasaka-Kim\cite{HIRASAKA-KIM}.
Non-Schurian Schur rings are of particular interest
since historically Wielandt\cite[Theorem~26.4]{WIELANDT}
found such an example, answering a question by Schur.
To avoid confusion, we use the term $S$-rings for Schur rings
(see Definition~\ref{def:S-ring} below).

The results for $S$-rings stated in the rest 
of this section are closely related with the results
by Evdokimov-Ponomarenko\cite{EV-NONSCHUR}. They used
generalized wreath products to construct non-Schurian
$S$-rings in a cyclic group, using delicate arguments.
We deal with only direct products of groups and the
usual wreath products, but still give a construction 
of non-Schurian $S$-rings.

We start with a definition of Cayley association schemes,
which is equivalent to the notion of $S$-rings.
We denote by $e$ the unit of a group.
We use the terminologies such as 
$S$-rings and Schurian $S$-rings according to
a survey by Muzychuk and Ponomarenko \cite{MUZYCHUK}.
The following definition of Cayley association
schemes is given in \cite[II.6]{Bannai} (without naming),
as well as the equivalence to the notion of $S$-rings.
We changed $g_2g_1^{-1}$ in the
definition there to $g_1^{-1}g_2$ because we consider 
the left action.
A detailed study on Cayley association schemes
and $S$-rings is found in \cite[Section~2.4]{CHEN-PON}.
\begin{definition} (Cayley association schemes)
\label{def:cayley}

Let $G$ be a finite group.
If there is a surjective mapping $r:G \to I$ such that the
composition 
$$
G \times G \to G \to I, \quad (g_1,g_2) \mapsto r(g_1^{-1}g_2)
$$
is an association scheme, then it is called a Cayley 
association scheme.
\end{definition}
This notion is equivalent to the following notion of
$S$-rings. The conditions on $r:G\to I$ are equivalent,
and the Bose-Mesner algebra of a Cayley
association scheme is naturally isomorphic
to the corresponding $S$-ring.
\begin{definition} ($S$-rings) \label{def:S-ring}

Let $G$ be a finite group, and $r:G \to I$ a surjective 
mapping. Let $\C[G]$ be the group ring. For a subset
$S\subset G$, define
$$\underline{S}:=\sum_{s\in S} s \in \C[G].$$
Let 
$$
A_r = \mbox{ the $\C$-linear span of } 
\underline{r^{-1}(i)} \subset \C[G]
\mbox{ for $i \in I$ }.
$$
Then $A_r$ is called an $S$-ring, if the following
conditions are satisfied.
\begin{enumerate}
 \item $\{e\}=r^{-1}(i_0)$ for some $i_0 \in I$.
 \item $A_r$ is closed under the product in $\C[G]$.
 \item For any $i\in I$, there is $i'\in I$
 with $\{g^{-1} \mid g \in r^{-1}(i)\}=r^{-1}(i')$.
\end{enumerate}
\end{definition}
We want to discuss on the Schurian property of $S$-rings.
\begin{definition} (Schurian $S$-rings)
\label{def:Schurian-ring}

Let $X$ be a finite set.
Let $\Gamma$ be a group transitively acting on $X$,
with a subgroup $G$ acting transitively and faithfully on $X$.
Fix $x\in X$.
Then we have a bijection
$$
G \to X, \quad g \mapsto gx,
$$
and through this bijection $\Gamma$ acts on $G$.
Let $\Gamma_e$ be the stabilizer of $e\in G$.
Consider
$$
r:G \to \Gamma_e\backslash G=:I.
$$ 
Then $r:G \to I$ gives an $S$-ring, which is called
a Schurian $S$-ring.
\end{definition}
The next proposition is well-known, and 
the proof is omitted.
\begin{proposition}\label{prop:equivalence}
Let $G$ be a finite group, and assume that 
$r:G \to I$ gives a Cayley association scheme,
and equivalently, an $S$-ring.
Then, the following are equivalent.
\begin{enumerate}
 \item\label{enum:cayley}
 The Cayley association scheme is Schurian 
 in the sense of Definition~\ref{def:schurian-sch}.
 \item\label{enum:ring}
 The $S$-ring is Schurian in the 
 sense of Definition~\ref{def:Schurian-ring}.
\end{enumerate} 
\end{proposition}
The following proposition is a direct consequence of the definitions 
and the equivalence between Cayley association schemes
and $S$-rings.
\begin{proposition}\label{prop:S-ring}
The wreath product of two $S$-rings
is an $S$-ring. 
The direct product of two $S$-rings
is an $S$-ring. 
\end{proposition}
\begin{proof}
Let $G_1$ and $G_2$ be finite groups.
Let $r_1:G_1\to I_1$ and $r_2:G_2 \to I_2$ 
be the corresponding mappings. Then, 
their wreath product is given 
by
\begin{center}
    \begin{tikzpicture}[auto]
    \node (a1) at (0,2) {$(G_{1}\times G_{2})\times (G_{1}\times G_{2})$};
    \node (a2) at (0,0) {$(G_{1}\times G_{1})\times (G_{2}\times G_{2})$};
    \node (b1) at (4,2) {$G_{1}\times G_{2}$};
    \node (b2) at (4,0) {$I_{1}\times I_{2}$};
    \node (c2) at (6.5,0) {$I_{1}\lexprod I_{2}$};
    \draw[->](a1) to node[swap] {} (a2);
    \draw[->](a1) to node {} (b1);
    \draw[->](a2) to node {} (b2);
    \draw[->](b1) to node {} (b2);
    \draw[->](b2) to node {} (c2);
    \end{tikzpicture},
\end{center}
where the top arrow is 
$
((g_1,g_2),(g_3,g_4))\mapsto (g_1,g_2)^{-1}(g_3,g_4)
=(g_1^{-1}g_3,g_2^{-1}g_4)
$.
The definition of the wreath product is via the left bottom corner.
By the commutativity of the diagram, it is an $S$-ring.
The claim for the direct product follows in a similar manner,
by merely removing $I_1\lexprod I_2$ from the above diagram.
\end{proof}
Proposition~\ref{prop:equivalence}, Proposition~\ref{prop:S-ring} 
and Theorem~\ref{th:schurian}
imply the following proposition.
\begin{proposition}
The wreath product of a non-Schurian $S$-ring
and an $S$-ring (in both order) is non-Schurian.
The same statement holds
for the direct product.
\end{proposition}
The existence of a large number of non-Schurian $S$-rings follows.
Using generalized wreath products, 
Evdokimov-Ponomarenko\cite{EV-NONSCHUR} proved the following
theorem.
\begin{theorem}\label{th:ev}
Let $n=p_1p_2p_3p_4n'$ be an integer where $p_1,p_2,p_3,p_4$ are prime numbers with 
the condition $\{p_1,p_2\}\cap\{p_3,p_4\}=\emptyset$ and $n'$ is a 
positive integer.
Put $d:=\lcm(p_1-1,p_2-1,p_3-1,p_4-1)$. If $d>2$, then 
the cyclic group of order $n$ has a non-Schurian $S$-ring. 
\end{theorem}

As another example, Hanaki-Hirai-Ponomarenko\cite{HANAKI-NONSCHUR} proved
a generalization of Wielandt's construction:
\begin{theorem}\label{th:hanaki}
Let $p$ be a prime.
Let $G$ be an elementary abelian $p$-group
of even rank except for the orders $2^2$, $3^2$, and $2^4$.
Then $G$ has a non-Schurian $S$-ring. 
\end{theorem}
Starting from these examples, by taking the wreath product with
any $S$-rings or association schemes,
we have a large family of non-Schurian $S$-rings
and non-Schurian association schemes.
\section{Iterated product and profinite association schemes}
\label{sec:profin}

A special case of an iterated wreath product 
is implicitly used in the construction of the kernel scheme
by Martin-Stinson\cite{MARTIN-STINSON}
(the notation here follows \cite{MOO}).
\begin{definition}\label{def:kernel}
Let $n$ be a positive integer, and $V$ a finite set of 
alphabet with cardinality $v\geq 2$. Let $X_n$ be $V^n$,
and $I_n:=\{1,2,\ldots,n\}\cup\{\infty\}$. (We use $\infty$
instead of a natural notation $n+1$, since this is $i_0$ and
to be distinguished when considering a projective system
in the next section.)
Define $R_n:X_n\times X_n \to I_n$ as follows.
Let $x=(x_1,x_2,\ldots,x_n)$ and $y=(y_1,y_2,\ldots,y_n)$ be
elements of $X_n$. Let $R(x,y)$ be the smallest index $i$
for which $x_i\neq y_i$. If $x=y$, then $R(x,y)=\infty$.
This is a symmetric (and hence commutative) 
association scheme, with $R^{-1}(\infty)$ being the identity
relation. This is called a kernel scheme, and denoted by
$\overrightarrow{k(n,v)}$.
\end{definition}
\begin{definition}
Let $X$ be an association scheme. 
We define inductively its wreath power for $n \in \N$ by:
\begin{itemize}
 \item $X^{\lexprod 1}:=X$.
 \item $X^{\lexprod n}:=(X^{\lexprod {n-1}})\lexprod X$ for $n\geq 2$.
\end{itemize}
\end{definition}

\begin{definition}\label{def:H1}
Let $v\geq 2$ be an integer. Let $H(1,v)$ be the class-one
association scheme of size $v$, namely, the unique association 
scheme with $I=\{i_0,i_1\}$ with $\#X=v$ (the notation 
follows Delsarte\cite[\S 4.1.1]{DELSARTE}).
\end{definition} 

\begin{proposition}\label{prop:kernel-prod}
\[
\overrightarrow{k(n,v)}=H(1,v)^{\lexprod n}.
\] 
\end{proposition}
\begin{proof}
This can be proved in a straightforward manner by induction on $n$.
The definition of $X\lexprod Y$ is 
``to determine the relation of $(x,y)$ and $(x',y')$,
first look at the $X$-component; if $x\neq x'$ then 
the relation is decided by them in $X$. If not,
then the relation is decided by those of $y$ and $y'$.''
This is compatible with the definition of the kernel schemes.
\end{proof}
Remark that in Martin-Stinson \cite{MARTIN-STINSON}, 
the kernel schemes are shown to be association schemes by computing the
intersection numbers. The above proposition gives
another proof.

\subsection{Profinite association schemes}
One of the motivations of this study 
is to construct a projective system of
non-Schurian association schemes. This section follows 
Matsumoto-Ogawa-Okuda\cite{MOO}.
Proofs of the statements are given there.
\begin{proposition}\label{prop:profin}
Let $X$ be an association scheme. Let $A_X$
be its Bose-Mesner algebra. We define 
a convolution product $\bullet$ on $A_X$ as a normalization 
of the matrix product
$$
A\bullet B:= \frac{1}{\#X}AB.
$$
\begin{enumerate}
\item
Let $p: X \to X'$ be a surjective morphism of 
association schemes. Through the identification 
of $A_X=C(I_X)$, we have a canonical 
injection $\Psi:A_{X'} \hookrightarrow A_X$ 
by $C(I_{X'}) \hookrightarrow C(I_X)$.
Then, $\Psi$ preserves Hadamard product, Hadamard unit,
and the convolution product $\bullet$
(does not preserve the convolution unit if $\#X>\#X'$).
\item
Suppose that $X$ and $X'$ are commutative.
Then, the set of primitive idempotents of $A_X$ 
with respect to the convolution product
is naturally identified with that to
the matrix product (and hence with $J(X)$).
The former set is obtained by multiplying
each element of $J(X)$ by $\#X$. From 
now on, $J(X)$ means the set of primitive 
idempotents with respect to $\bullet$.
\item\label{enum:parto}
An element of $J(X')$ is mapped by $\Psi$ to a non-zero idempotent
in $A_X$, and thus a non-empty sum of elements of $J(X)$.
For distinct elements of $J(X')$, the corresponding
non-empty subsets of $J(X)$ have no intersection.
This gives a one-to-many (and non-empty) 
correspondence $J(X') \to J(X)$, in other words, a partial
surjection from $J(X)$ to $J(X')$, denoted by $J(X) \parto J(X')$.
\end{enumerate}

\end{proposition}
\begin{definition}
Let $\Lambda$ be a directed ordered set, namely,
a partial ordered set where any two elements have an upper bound.
A profinite association scheme $(X_\lambda)_{\lambda\in \Lambda}$
is a projective system
of association schemes with surjective morphisms, namely:
\begin{enumerate}
\item
A family of association schemes $(R_\lambda, X_\lambda, I_\lambda)$
for $\lambda\in \Lambda$. 
\item
For any $\lambda \geq \mu \in \Lambda$, a surjective morphism
$p_{\lambda,\mu}:X_\lambda \to X_\mu$ is specified.
\item For any $\lambda$, $p_{\lambda,\lambda}=\id_{X_\lambda}$.
\item For any $\lambda \geq \mu \geq \nu$,
$$
p_{\lambda,\nu}=p_{\mu,\nu}\circ p_{\lambda,\mu}.
$$
\end{enumerate}
We define its underlying set by 
$$
\Xhat=\varprojlim X_\lambda,
$$
and the set of relations by
$$
\Ihat=\varprojlim I_\lambda,
$$
in the category of sets.
Then, $\Xhat$ and $\Ihat$ have natural (profinite)
topologies, where $X_\lambda$ and $I_\lambda$ are
finite sets with the discrete topology.
If every $X_\lambda$ is commutative, we have a 
projective system of partial surjections
of $J_\lambda:=J(X_\lambda)$. We define
$$
\Jhat =\varprojlim J_\lambda,
$$
which is proved to have a discrete topology.
We define the Bose-Mesner algebra of $(X_\lambda)_{\lambda \in\Lambda}$
as the inductive limit
$$
A_{\Xhat}:=\varinjlim A_{X_\lambda},
$$
which has Hadamard product with unit, and the convolution 
product (without unit if $\Jhat$ is infinite). 
It has a linear basis $\Jhat$ consisting
of all the primitive idempotents, and is isomorphic, as a 
ring with Hadamard product, to the space of 
locally constant functions on $\Ihat$.
For subsets $J_D\subset \Jhat$, $I_C\subset \Ihat$,
and a finite multi-subset $Y\subset \Xhat$,  
a property ``$I_C$-free code and $J_D$-design''
of $Y$ is defined.
\end{definition}

\subsection{Iterated wreath products}
Iterated wreath products give examples of 
profinite association schemes. 
We begin with preparation.
\begin{lemma} The wreath product is associative, i.e.,
there is a canonical isomorphism
$$
(X\lexprod Y)\lexprod Z \to X\lexprod (Y\lexprod Z)
$$
for association schemes $X,Y$ and $Z$.
Hence we may write
$$
X\lexprod Y\lexprod Z.
$$
\end{lemma}
\begin{proof}
This follows from the identification
\begin{eqnarray*}
(I_X\lexprod I_Y)\lexprod I_Z 
&=&([(I_X\setminus \{i_0\})\coprod I_Y] \setminus \{i_0\}) \coprod I_Z \\
&=& (I_X\setminus \{i_0\})\coprod (I_Y \setminus \{i_0\}) \coprod I_Z \\
&=& I_X\lexprod (I_Y\lexprod I_Z)
\end{eqnarray*}
and the commutativity of 
\[
    \begin{aligned}
      \begin{tikzpicture}[auto]
      \node (a1) at (0,2) {$I_X\times I_Y \times I_Z$};
      \node (a2) at (0,0) {$(I_X\lexprod I_Y)\times I_Z$};
      \node (b1) at (5,2) {$I_X\times (I_Y\lexprod I_Z)$};
      \node (b2) at (5,0) {$(I_X\lexprod I_Y)\lexprod I_Z=I_X\lexprod (I_Y\lexprod I_Z)$.};
      \draw[->](a1) to node[swap] {$\lex \times \id$} (a2);
      \draw[->](a1) to node {$\id \times \lex$} (b1);
      \draw[->](a2) to node {$\lex$} (b2);
      \draw[->](b1) to node {$\lex$} (b2);
      \end{tikzpicture}
    \end{aligned}
\]
\end{proof}

\begin{lemma}\label{lem:image-A}
Let $X$ and $Y$ be association schemes.
The projection $\pi:X \lexprod Y \to X$
given in Proposition~\ref{prop:proj} induces an
injection 
$$
A_X \to A_{X\lexprod Y}, \quad A \mapsto A \otimes \bJ_Y.
$$
\end{lemma}
\begin{proof}
Because of the definition of 
$$
\pi_{I_X}:I_X \lexprod I_Y \to I_X
$$
in Definition~\ref{def:functor}, the preimage of 
$i_X \in I_X$ in $I_X\times I_Y$ is 
$$
\{(i_X,i_Y) \mid i_Y \in I_Y\},
$$
and hence the image of $A_{i_X}\in A_X$ in $A_{X\lexprod Y}$
is 
$$
\sum_{i_Y\in I_Y}A_{i_X} \otimes A_{i_Y}=A_{i_X}\otimes \bJ_Y
$$
by Proposition~\ref{prop:direct}.
Since $A_X$ is the linear span of $A_{i_X}$, 
the statement follows.
\end{proof}

\begin{proposition}\label{prop:finite}
Let $X_1,X_2,\ldots,X_n$ be any sequence of 
association schemes. Let $I_1,I_2,\ldots,I_n$ be
the set of their relations.
Then, their wreath product
$X_1\lexprod X_2\lexprod X_3 \lexprod \cdots \lexprod X_n$
has underlying set
$$
X_1\times \cdots \times X_n
$$
and the set of relations
$$
I_1\lexprod I_2\lexprod I_3\lexprod \cdots \lexprod I_n
= (I_1\setminus\{i_0\})\coprod(I_2\setminus\{i_0\})\coprod
\cdots\coprod(I_{n-1}\setminus\{i_0\})\coprod I_n.
$$
\end{proposition}
\begin{proof}
The structure of the underlying set follows by definition.
The structure of the set of relations
follows by induction from Definition~\ref{def:lexI}. 
\end{proof}

\begin{proposition}\label{prop:lim}
Let $X_1,X_2,\ldots,$ be an infinite
series of association schemes.
Then, the series of the wreath products 
$$
(X_1\lexprod\cdots\lexprod X_n)_{n\in \N_{>0}}
$$
form a projective system of association schemes.
The mappings of the underlying sets are given by 
projections
$$
X_1\times \cdots \times X_n \to X_1\times \cdots \times X_{n-1}.
$$
The mappings of the sets of relations
$$
(I_1\setminus\{i_0\})\coprod
\cdots\coprod(I_{n-1}\setminus\{i_0\})\coprod I_n
\to
(I_1\setminus\{i_0\})\coprod
\cdots\coprod(I_{n-2}\setminus\{i_0\})\coprod I_{n-1}
$$
are given by mapping the elements in $I_n$
to the $i_0$ in $I_{n-1}$.

The projective limit of the underlying set is
the direct product (with direct product topology, hence
compact and Hausdorff)
$$
\Xhat=\prod_{i=1}^\infty X_i.
$$
The projective limit $\Ihat$ of $I_1\lexprod\cdots\lexprod I_n$
is the one-point compactification of
the discrete topological set
$$
\coprod_{i=1}^\infty(I_i\setminus\{i_0\}).
$$
\end{proposition}
\begin{proof}
The mapping between the underlying set is the projection 
by Proposition~\ref{prop:proj}.
It is a general
fact that the projective limit of finite direct products
is the infinite direct product.

The projective system $(I_1\lexprod\cdots \lexprod I_n)_{n\in \N_{>0}}$ 
is given by mapping the last $I_n$
to the $i_0$ of $I_{n-1}$, by Proposition~\ref{prop:proj}.
We consider its projective limit.
Except for $i_0$, every element in the coproduct
is a clopen point in the projective limit, and the 
set of open neighborhoods of the limit of $i_0$ is
the set of the union of $\{i_0\}$ and 
the complement of a finite set of 
$\coprod_{i=1}^\infty(I_i\setminus\{i_0\})$.
\end{proof}
\begin{proposition}\label{prop:ind-lim}
Suppose that every $X_n$ is commutative in 
Proposition~\ref{prop:lim}. Then, the primitive idempotents
of 
$X_1\lexprod X_2\lexprod X_3\cdots \lexprod X_n$
is
\begin{equation}\label{eq:J-coprod}
J_{X_1}\coprod (J_{X_2}\setminus\{j_0\})\coprod (J_{X_3}\setminus\{j_0\})
\coprod \cdots \coprod (J_{X_n}\setminus\{j_0\}).
\end{equation}
Its inductive limit is
\begin{equation}\label{eq:J-inf}
\Jhat:=J_{X_1}\coprod \coprod_{i=2}^\infty
(J_{X_i}\setminus\{j_0\}).
\end{equation}
\end{proposition}
\begin{proof}
By induction using Corollary~\ref{cor:lex-J},
(\ref{eq:J-coprod}) is equal to
$J(X_1\lexprod X_2\lexprod X_3\cdots \lexprod X_n)$.
By Lemma~\ref{lem:image-A}
the one-to-many correspondence (\ref{enum:parto}) 
in Proposition~\ref{prop:profin}
\begin{equation}\label{eq:J-n}
J(X_1\lexprod X_2\lexprod X_3\lexprod \cdots \lexprod X_{n-1})
\to 
J(X_1\lexprod X_2\lexprod X_3\lexprod \cdots \lexprod X_{n})
\end{equation}
is given by 
$$E \mapsto E \otimes \bJ_{X_n},
$$
where the right-hand side is 
a primitive idempotent (w.r.t. $\bullet$) in 
$A_{X_1\lexprod X_2\lexprod X_3 \cdots \lexprod X_{n}}$
by Theorem~\ref{th:main}. This is
a natural inclusion of (\ref{eq:J-coprod}) for $n-1$
to that for $n$.
Thus, the partial surjection
(\ref{enum:parto}) in Proposition~\ref{prop:profin}, namely, 
$$
J(X_1\lexprod X_2\lexprod X_3\lexprod\cdots \lexprod X_{n})
\parto 
J(X_1\lexprod X_2\lexprod X_3\lexprod\cdots \lexprod X_{n-1})
$$
is induced by the natural inclusion (\ref{eq:J-n}).
The projective limit of the partial surjections is in this case equal
to the inductive limit of injections, hence is a union 
(\ref{eq:J-inf}).
\end{proof}
By Proposition~\ref{prop:kernel-prod}, the kernel schemes
in Definition~\ref{def:kernel} is a special case of
Proposition~\ref{prop:lim}, where each $I_i$ (and consequently $J_i$) 
has the cardinality two. They form a projective system,
where $I_n \to I_{n-1}$ is mapping $i\mapsto i$ for $i<n$, 
$n\mapsto \infty$, and $\infty \mapsto \infty$, 
as proved in Proposition~\ref{prop:lim}.  
$J_{n-1} \to J_n$ is a canonical inclusion,
as proved in Proposition~\ref{prop:ind-lim}.  
The above iterated wreath products give examples
of profinite association schemes whose $\Xhat$, $\Ihat$,
and $\Jhat$ are explicitly described. 
There is a closely related earlier research
by Barg and Skriganov \cite[Section~8]{BARG-SKRIGANOV},
where they treat similar objects coming from
a profinite abelian group, and obtain the duality theorems and
the structural constants.

Our final remark is about a relation with Kurihara-Okuda\cite{KURIHARA-OKUDA}.
There, for any compact Hausdorff group $G$ and its closed subgroup $H$,
the notion of Bose-Mesner algebra for the homogeneous 
space $G/H$ is given (which may be seen as an analogue to a Schurian scheme).

Any profinite group $G$ is compact and Hausdorff, and 
for any closed subgroup $H$, $G/H$ can be viewed as both
a homogeneous space (as in \cite{KURIHARA-OKUDA})
and a profinite association scheme as in \cite{MOO}.
Both methods yield the same Bose-Mesner algebra.
In this case, $G/H$ yields a projective system of
Schurian association schemes.
Theorem~\ref{th:schurian} and the iterated 
wreath products imply that there is a large class of 
projective systems of finite non-Schurian association schemes.
\section*{Acknowledgment}
We would like to show our best gratitude to Professor
Ilia Ponomarenko, who kindly informed us of appropriate references
related to the origin of the wreath products, and the 
adequate existing researches on the wreath products 
of coherent configurations (hence of association schemes),
in particular \cite{CHEN-PON} that
covers Section~\ref{sec:const-non-sch}
with more systematic and detailed discussions, after he found an earlier
version of this article in the arXiv.

\bibliographystyle{plain}
\bibliography{sfmt-kanren}


\end{document}